\documentclass[english,12pt]{smfart}
  \usepackage[utf8]{inputenc}
  \usepackage{smfthm,smfenum}
  \usepackage{amsmath,amsxtra,amscd,amssymb,latexsym,stmaryrd}
  \usepackage{mathrsfs,dsfont}
  \usepackage{graphicx}
  \usepackage{color}
  \usepackage{pinlabel}

\title[Distance graphs]{Coloring distance graphs: a few answers and many questions}

\author{Beno\^{\i}t R. Kloeckner}
\address{Universit\'e de Grenoble I, Institut Fourier\\ CNRS UMR 5582\\ BP 74\\
  38402 Saint Martin d'H\`eres cedex\\ France}
\email{benoit.kloeckner@ujf-grenoble.fr}

\theoremstyle{plain}
\newtheorem{prob}{Problem}

\newcommand{\planar}{\mathop{\mathsc{planar}}\nolimits}
\newcommand{\interval}{\mathop{\mathsc{interval}}\nolimits}
\newcommand{\compact}{\mathop{\mathsc{compact}}\nolimits}
\newcommand{\packing}{\mathop{\mathrm{Pack}}\nolimits}
\newcommand{\clique}{\mathop{\mathrm{K}}\nolimits}

\newcommand{\mathsc}[1]{{\normalfont\textsc{#1}}}

\begin{document}
\begin{abstract}
Given a metric space and a set of distances, one constructs the associated
distance graph by taking as vertices the points of the space and as edges
the pairs whose distance is in the given set.

It is a longstanding open question to determine the chromatic number of
the graph obtained from the Euclidean plane and a set reduced to one distance.

Surprisingly, while many variants of this problem have been studied, 
only a few non-Euclidean spaces seem to have been seriously considered.
In this paper, we consider the planar translation-invariant metrics
and the hyperbolic plane. We answer questions of Johnson and Szlam,
prove a few other results, and ask many questions.
\end{abstract}

\maketitle

Let $X$ be a metric space, whose metric (or distance function) shall be denoted by $\rho$,
and $D$ be a subset of $\mathbb{R}^{+*}$. Often, $D$ will be a singleton
and in some cases, one can restrict to $D=\{1\}$.

The \emph{distance graph} defined by $X$ and $D$ is defined as the graph
$G(X,D)$ whose vertices are the elements of $X$ and where $(xy)$ is an
edge exactly when $\rho(x,y)\in D$.

Let us say a word about the definition of the chromatic number,
because we will occasionally deal with the case when it is infinite, and
prefer not to dwell on cardinals and axioms.
As usual, the graph $G$ has \emph{chromatic number} $\chi(G)=k\in\mathbb{N}$
if there is a $k$-coloring of $G$, i.e. a map from the vertices
of $G$ to $\{1,\dots,k\}$ that maps adjacent vertices to different values,
and if $k$ is minimal with this property.
Now, we shall say that $\chi(G)\leq \aleph_0$ if there is a coloring with values in
$\mathbb{N}$, and $\chi(G)=\aleph_0$ if in addition there is no $k$-coloring for any finite
$k$.

The graph $G(X,D)$ is usually not locally finite,
but it turns out that in many interesting cases it has finite chromatic number.
Thanks to a theorem of De Bruijn-Erd\"os \cite{DeBruijn-Erdos},
under the axiom of choice a graph is $k$-colorable if and only if all
its finite subgraphs are $k$-colorable. Note that the chromatic number of
$G(X,\{1\})$ may depend on the chosen set of axioms in some cases,
see e.g. \cite{Soifer-Shelah}. The chromatic number of $G(X,D)$
shall be denoted by $\chi(X,D)$. Every now and then, we shall find convenient
to use $\rho$ instead of $X$ (notably, when the underlying set of $X$ is fixed but
the metric changes) in some notations.

A long-standing open question (see several chapters of \cite{Soifer:proc}) 
is to determine the chromatic number $\chi(\mathbb{E}^2,\{1\})$
where $\mathbb{E}^2$ is the Euclidean plane $\mathbb{R}^2$ with the usual distance.
We know little more than what was proved soon after the question was raised:
\[4\leq \chi(\mathbb{E}^2,\{1\})\leq 7.\]
More recently, it was proved that Lebesgue-measurable colorings of $G(\mathbb{E}^2,\{1\})$
must have at least $5$ colors \cite{soifer:book,Falconer}; in particular,
$5\leq\chi(\mathbb{E}^2,\{1\})$
under replacement axioms for the axiom of choice including
the axiom that all subsets of $\mathbb{R}$ are
Lebesgue measurable.

The higher dimensional Euclidean spaces have been well considered \cite{Soifer:proc},
as well as variations (see e.g. \cite{Soifer:Soifer} in the above mentioned reference) 
but little attention has been given to the case when $X$ is not Euclidean.
Simmons considered the round spheres \cite{Simmons};
Chilakamarri \cite{Chilakamarri} considered the Minkowski planes 
(i.e. $\mathbb{R}^2$ endowed with a norm)
and showed interesting relationship with the Euclidean case; and \cite{Johnson-Szlam} briefly
considered the more general case when $\mathbb{R}^2$ is endowed with a translation-invariant
distance that induces the usual topology. Johnson and Szlam
 asked several questions, some of which we answer
in Section \ref{sec:translation}. Our answers are pretty simple and show that the questions
in \cite{Johnson-Szlam} may be in fact too flexible; we therefore propose a variation and
give a partial answer, where prime numbers make an intriguing appearance (or rather,
disappearance). We also propose another variation on this question.

Then, we move further away from Euclideanness
by studying in Section \ref{sec:hyperbolic} the case when $X=\mathbb{H}^2$,
the hyperbolic plane. This case has been suggested by Matthew Kahle
on MathOverflow \cite{Kahle}.
 As $\mathbb{H}^2$ has no homothety, even
when $D$ is a singleton the choice of the value matters a lot;
as a consequence, the question of determining the behavior of
$\chi(\mathbb{H}^2,\{d\})$ when $d$ varies seems as rich as
the question of the relation between $\chi(\mathbb{E}^n,\{1\})$
and $n$.

\section{Translation-invariant metric on the plane}\label{sec:translation}

In this section, we consider $X=(\mathbb{R}^2,\rho)$ where
$\rho$ is a metric (a positive, symmetric, definite function on
$(\mathbb{R}^2)^2$ that satisfies the triangular inequality). To allow $\rho$ to be completely
general would let way too much flexibility, as any metric space
with the cardinality of $\mathbb{R}$ could be considered.

\subsection{The Johnson-Szlam problem}

Following Johnson and Szlam \cite{Johnson-Szlam}, we shall therefore ask that $\rho$ is
translation invariant:
\[\rho(x,y)=\rho(x+v,y+v) \quad\forall x,y,v\in \mathbb{R}^2\]
and that $\rho$ induces the usual topology on $\mathbb{R}^2$, which
means that all Euclidean balls contain a $\rho$-ball of same center
(all balls are assumed to have positive radius) and vice-versa.
We shall call such a $\rho$ a \emph{plane metric} and
$(\mathbb{R}^2,\rho)$ a \emph{planar metric space}.

Note that by \emph{open ball} of $\rho$ we mean the sets
\[B_\rho^\circ(x,r) = \{y \,|\, \rho(x,y)< r\}\]
and by closed balls the
\[B_\rho(x,r) = \{y \,|\, \rho(x,y)\leq r\};\]
please beware that in general $B_\rho^\circ(x,r)$ needs not be the interior of $B_\rho(x,r)$.

In \cite{Johnson-Szlam}, Johnson and Szlam pose the problem of
determining the set $\planar$ of all possible values of $\chi(X,\{1\})$ when $X$ runs over
all planar metric spaces, and show that $3\in \planar$. 
In \cite{Chilakamarri} it is shown that $4\in\planar$ (realized by the supremum norm).
As subquestions, Johnson and Szlam asked whether $2\in\planar$
and whether $\planar \leq \chi(\mathbb{E}^2,\{1\})$ (by $D\leq r$ we mean that all elements
of the set $D$ are lesser than or equal to the number $r$). We shall answer all these questions
by proving, at the end of this section, the following result (we use the anglo-saxon convention that
$\mathbb{N}$ does not contain $0$).
\begin{theo}\label{theo:planar}
$\planar = \mathbb{N}\cup\{\aleph_0\}$.
\end{theo}

To this end, we use the following facts to construct ad hoc metrics.
\begin{enumerate}
\item if $\rho$ is a metric, so is $\displaystyle\frac{\rho}{1+\rho}$, 
\item if $\rho$ is a metric, so is $\min(\rho,r)$ for any positive number $r$,
\item if $\rho$ and $\rho'$ are metrics on spaces $X,X'$
      then 
      \begin{eqnarray*}
        \rho \stackrel{\infty}{\times} \rho'&:=& ((x,x'),(y,y'))\mapsto \max(\rho(x,y),\rho'(x',y'))
      \end{eqnarray*}
      is a metric on $X\times X'$,
\item if $\rho$ and $\rho'$ are metrics on $X$, then so is $\max(\rho,\rho')$.
\end{enumerate}
All of this is elementary and most is classical. The first item follows from
concavity of the chosen function, the second from a case analysis and
the last item follows from the penultimate.

Let us start small.
\begin{prop}\label{prop:12planar}
$\{1,2\}\subset\planar$
\end{prop}

\begin{proof}
Define a metric on $\mathbb{R}^2$ by
\[\rho_1(x,y) = \frac{|x-y|}{1+|x-y|}\]
where $|\cdot|$ is the Euclidean norm.
It is clearly translation invariant, defines the usual topology and no two points
are at distance $1$ so $\chi(\rho_1,\{1\})=1$.

Define a metric on $\mathbb{R}^2$ by
\[\rho_2(x,y) = \max(|x_1-y_1|,\frac{|x_2-y_2|}{1+|x_2-y_2|})\]
where $(x_1,x_2)=x$ and $(y_1,y_2)=y$.
It is clearly translation invariant, defines the usual topology and
two points $x,y$ are at distance $1$ if and only if $|x_1-y_1|=1$.
It follows first that such pairs of points exist, so that $\chi(\rho_2,\{1\})\geq 2$,
second that a coloring of the plane in alternating blue and red semi-open vertical strips
of Euclidean width $1$ avoids any monochromatic pairs of point at $\rho_2$-distance $1$,
so that $\chi(\rho_2,\{1\})\leq 2$.
\end{proof}

Now, let us look at the other end of the spectrum.
\begin{prop}\label{prop:aleph0}
$\aleph_0\in\planar$ and $\planar\leq \aleph_0$.
\end{prop}

\begin{proof}
Define a metric on $\mathbb{R}^2$ by
\[\rho_\infty(x,y) = \min(|x-y|, 1).\]

A tiling of the plane by squares of Euclidean diagonal less than $1$ gives a coloring
with countably many colors (map each point to the square it is on) such that any pair
of points at $\rho_\infty$-distance $1$ must have different colors, and we see that
$\chi(\rho_\infty,\{1\})\leq \aleph_0$.

Up to replacing $1$ with another value, this argument in fact uses
only the property that $\rho_\infty$ induces the usual topology and is translation-invariant:
it therefore extends to every planar metric, and $\planar\leq \aleph_0$.

Finally, for all $k\in\mathbb{N}$, one can find $k$ points with pairwise
Euclidean distance at least $1$ (put them regularly on a large enough Euclidean circle).
These points are pairwise at $\rho_\infty$-distance $1$ so that $\chi(\rho_\infty,\{1\})\geq k$
for all finite $k$, and $\chi(\rho_\infty,\{1\})=\aleph_0$.
\end{proof}

The rest of the proof of Theorem \ref{theo:planar} is postponed until the end of Section 
\ref{sec:translation}.

\subsection{A variation on the Johnson-Szlam problem}

The above answers to Johnson and Szlam's questions are in some sense very degenerate, and one wonders
if there is a natural assumption one could further ask $\rho$ to satisfy. There is one indeed:
asking $\rho$ to induce the same topology is linked to its small open balls,
and a somewhat dual condition
is to ask it to be \emph{proper}, that is asking its closed balls to be compact.
Since closed balls of $\rho$ are closed, this amounts to ask them to be bounded in the Euclidean metric
(so that $\rho(x,y)\to\infty$ when $x$ is fixed and $|y|\to\infty$).

\begin{prob}\label{prob:proper}
Determine the set $\planar^*$ of all possible values of $\chi(X,\{1\})$ when $X$ runs over
all proper planar metric spaces.
\end{prob}

The above examples only show that $4\in\planar$, as only the supremum norm is proper
among them. Let us give two pieces of information with regard to this problem; first
we show that Propositions \ref{prop:12planar} and \ref{prop:aleph0}
do not hold anymore for proper planar metrics.

\begin{theo}\label{theo:proper_large}
$\planar^*>2$.
\end{theo}

\begin{proof}
Let $\rho$ be a proper planar metric and $S=S_\rho(0,1)$ be the unit $\rho$-sphere centered
at $0$. By properness, $S$ is not empty (it contains at least the topological
frontier $\partial B$ of $B=B_\rho(0,1)$, which cannot be empty).
This already shows $\planar^*>1$.

Let $y$ be a point such that $|y|>|x|$ for all $x\in B$.
It may happen that $S$ is not connected; but since $S$ disconnects $0$ from 
$y$, there is a connected component $S'$ of $S$ that disconnects
$0$ from $y$. This is not a triviality, and can be found e.g. in \cite{Newman}
Section 14, see also the MathOverflow question and answer \cite{B-T}.

Moreover, by translation invariance and symmetry of $\rho$, $S$ is symmetric
with respect to $0$. It follows that $-S'$ is a connected subset
of $S$ that disconnects $0$ from $-y$; but $y$ and $-y$ are connected outside $B$,
so that both $S'$ and $-S'$ disconnnect $0$ from $y$.

Let $x\in S'$ be a point maximizing $|x|$ and
let $S''=-S'+x$ be
the $x$ translate of $-S'$; then $S''$ contains $0$ and a point $z$ 
such that $|z|>|x|$
(take $z=x+rx$ where $r>0$ is such that $-rx\in S'$).

By the definition of $x$, $z$ is connected to $y$
outside of $S'$, and it follows that $S'$ must disconnect $z$ from $0$.
Since $S''$ is connected, there must be a point $w\in S'\cap S''$.
Now, the points $0,x,w$ form a triangle
of the graph $G(\rho,\{1\})$, which must therefore have chromatic number at least $3$.
\end{proof}

\begin{theo}\label{theo:proper_finite}
$\planar^*\subset\mathbb{N}$.
\end{theo}

\begin{proof}
Let $\rho$ be a proper planar metric, and let us prove that $\chi(\rho,\{1\})$ is finite.

Choose $\varepsilon>0$ small enough that all pairs of points in the square $s$ 
of side length $\varepsilon$ centered at $0$ are at $\rho$-distance less than $1$,
and choose $E$ large enough 
that the closed $\rho$-ball of radius $1$ is contained in the square $S$ entered at $0$
of side length $E-\varepsilon$. Up to enlarge $E$, assume that $E=n\varepsilon$ for some integer
$n$, and tile $\mathbb{R}^2$ with translates of $s$ by elements of $\varepsilon\mathbb{Z}^2$
(in order to get a genuine partition of the plane, we shall assign to each square the points
of its closed left side and of its open lower side).

Now, we label each tile $s+\varepsilon(a,b)$ where $a,b\in\mathbb{Z}$ with the residues
of $(a,b)$ modulo $n$, and color each point of the plane by its tile's label. 
By construction, pairs of points at $\rho$-distance $1$ must be colored with different colors,
so that $\chi(\rho,\{1\})\leq n^2$.
\end{proof}

Next, we show that Problem \ref{prob:proper} can be related to the Euclidean case.

\begin{prop}\label{prop:proper_plane}
For all $d>1$, we have
\[\chi(\mathbb{E}^2,[1,d])\in \planar^*.\]
\end{prop}

\begin{proof}
Define a metric on $\mathbb{R}^2$ by
\[\rho^*_d(x,y) = \max\left(\min(|x-y|,1),\frac1d|x-y|\right).\]
It is clearly translation-invariant and it defines the usual topology.

Its spheres are either Euclidean circles, when the radius is different from $1$,
or the annulus
\[\{y\in\mathbb{R}^2 \,|\, |x-y|\in [1,d]\}=S_\rho(x,1).\]
In particular $\rho_d^*$ is proper and
$G((\mathbb{R}^2,\rho_d^*),1)=G(\mathbb{E}^2,[1,d])$.
\end{proof}

Proposition \ref{prop:proper_plane} 
motivates us to estimate $\chi(\mathbb{E}^2,[1,d])$ in terms of $d$, a question
worth asking even without the non-Euclidean context in mind.

\begin{prob}
Determine the set $\interval$ of values of $\chi(\mathbb{E}^2,[1,d])$ when $d>1$.
\end{prob}

One could want to be even more specific.
\begin{prob}\label{prob:interval}
Determine the function $d\mapsto\chi(\mathbb{E}^2,[1,d])$ defined for $d>1$.
\end{prob}

This problem might be too ambitious, so let us consider further 
sub-problems, first in the small $d$ regime.

\subsection{On the plane with a small interval of distances}

\begin{prob}
Determine $\displaystyle\chi(\mathbb{E}^2,1+):=\lim_{d\to 1,d>1} \chi(\mathbb{E}^2,[1,d])$.
\end{prob}
Note that the limit exists since $\chi(\mathbb{E}^2,[1,d])$ is non-decreasing in $d$.
Moreover the proofs of the known estimates for $\chi(\mathbb{E}^2,\{1\})$
apply readily to this case, so $4\leq \chi(\mathbb{E}^2,1+)\leq 7$; but maybe this problem is
more tractable than to determine the chromatic number of the plane.
Even without determining the values of $\chi(\mathbb{E}^2,\{1\})$ or $\chi(\mathbb{E}^2,1+)$,
one wonders:
\begin{prob}\label{prob:1+}
Do we have $\chi(\mathbb{E}^2,1+)>\chi(\mathbb{E}^2,\{1\})$?
\end{prob}
A positive answer would be a very strong result, as
it would imply $\chi(\mathbb{E}^2,\{1\})<7$; but a negative answer would tell that one
can determine $\chi(\mathbb{E}^2,\{1\})$ by considering graphs with vertices \emph{almost}
$1$ apart, a great flexibility! However answering Problem \ref{prob:1+} may be as difficult as
determining $\chi(\mathbb{E}^2,\{1\})$.

This question has been considered by Exoo \cite{Exoo}, who proves that:
\begin{align*}
\mbox{for } & 1.3114\dots < d < 1.3228\dots,&  \chi(\mathbb{E}^2,[1,d])=7 \\
\mbox{for } & 1.0172\dots < d, &  \chi(\mathbb{E}^2,[1,d])\geq 5
\end{align*}
and conjectures $\chi(\mathbb{E}^2,1+)=7$.

One can already ask the following, less ambitious question.
\begin{prob}
Is $\chi(\mathbb{E}^2,1+)\geq 5$?
\end{prob}

\subsection{On the plane with a large interval of distances}

Let us see what can easily be said in the high $d$ regime.
\begin{prop}\label{prop:interval}
For all $d>1$, we have
\[\packing(1,d+1) \leq \chi(\mathbb{E}^2,[1,d]) \leq \lceil\sqrt{2}d+1\rceil^2\]
where $\packing(1,r)$ is the maximal number of discs of diameter $1$
that can be packed in a disc of diameter $r$.
\end{prop}

\begin{proof}
The lower bound is obvious, since one can find $\packing(1,d+1)$ points in the plane
whose pairwise distance are in $[1,d]$ (take the centers of a packing).

The upper bound is obtained as in Theorem \ref{theo:proper_finite}
with $\varepsilon=1/\sqrt{2}$ and $n=\lceil\sqrt{2}d+1\rceil$.
\end{proof}
It seems inefficient to use a $\mathbb{Z}/n$-indexed square tiling for the upper bound, and
indeed Exoo \cite{Exoo} gets a better bound using an hexagonal grid.
It is well-known that 
\[\packing(1,d+1)\sim \frac{\pi}{\sqrt{12}} d^2\]
when $d\to\infty$, and we are led to the following problem.
\begin{prob}
Find the values of 
\[\limsup_{d\to\infty} \frac{\chi(\mathbb{E}^2,[1,d])}{d^2} \quad\mbox{and}\quad
\liminf_{d\to\infty} \frac{\chi(\mathbb{E}^2,[1,d])}{d^2}.\]
\end{prob}

Proposition \ref{prop:interval} shows
\[\liminf_{d\to\infty} \frac{\chi(\mathbb{E}^2,[1,d])}{d^2} \geq \frac{\pi}{\sqrt{12}}\simeq 0.9069\]
and Exoo's Theorem 2 implies
\[\limsup_{d\to\infty} \frac{\chi(\mathbb{E}^2,[1,d])}{d^2} \leq \frac{4}{3}.\]
While I am not confident enough to make a precise conjecture, my personnal guess is that
one could extend Sz\'ekely's proof of $\chi(\mathbb{E}^2,\{1\})\leq 7$ \cite{Szekely}
to improve the upper bound.

In case a precise determination could not be reached, one could ask the following.
\begin{prob}
Does $\displaystyle\frac{\chi(\mathbb{E}^2,[1,d])}{d^2}$ have a limit when $d$ goes to infinity?
\end{prob}

\subsection{Possible chromatic numbers of proper planar metrics}

Back to the question of determining the set $\planar^*$, all what precedes
gives some information: there must be elements of $\planar^*$
in any integer interval of the form $\llbracket n, \alpha n \rrbracket$
for $\alpha>\frac{8}{\pi\sqrt{3}}$ and $n$ large enough. But this is weak as it
only bounds below the cardinal of $\planar^*\cap\llbracket 1,N\rrbracket$
by a multiple of $\log N$; in fact most integers are in $\planar^*$.

\begin{theo}\label{theo:proper}
$\planar^*$ contains all \emph{non-prime} integers
greater than $1$.
\end{theo}

We shall construct ad-hoc metrics using the supremum product metric.

\begin{lemm}
Let $(X_i,\rho_i)$ ($i=1,2$) be two metric spaces. Then
we have
\[ \max_i(\chi(\rho_i,\{1\})) \leq \chi(\rho_1 \stackrel{\infty}{\times} \rho_2,\{1\})
   \leq \chi(\rho_1,\{1\}) \cdot \chi(\rho_2,\{1\})\]
and
\[\clique(G(\rho_1,\{1\}))\times\clique(G(\rho_2,\{1\})) 
  \leq \clique(G(\rho_1 \stackrel{\infty}{\times} \rho_2,\{1\}))\]
where $\clique(\rho, D)$ is the size of a largest clique
in $G((X,\rho),D)$.
\end{lemm}

\begin{proof}
Write $X=X_1\times X_2$ and $\rho=\rho_1 \stackrel{\infty}{\times} \rho_2$.
The lower bound on $\chi(\rho,\{1\})$ is obvious, as $X$ contains isometric copies
of $X_1$ and $X_2$. The upper bound is obtained by taking colorings $c_i$ of the $X_i$
and defining $c(x_1,x_2)=(c_1(x_1),c_2(x_2))$. If $x=(x_1,x_2)$ and $y=(y_1,y_2)$
are at $\rho$-distance $1$, then $\rho_i(x_i,y_i)=1$ for some $i$. Then $c_i(x_i)\neq c_i(y_i)$
and therefore $c(x)\neq c(y)$.

The lower bound on $\clique$ is also obvious, as the strong product of cliques of
$G(X_1,\{1\})$ and $G(X_2,\{1\})$ is a clique in $G(X,\{1\})$.
\end{proof}

\begin{lemm}
We have 
\[\lfloor d\rfloor +1 \leq \clique(G(\mathbb{E}^1,(1,d])) \quad\mbox{and}\quad
  \chi(\mathbb{E}^1,[1,d]) \leq \lceil d \rceil +1.\]
\end{lemm}

\begin{proof}
The lower bound on the clique number is obtained by taking the $k$ vertices
$0,1,\dots,k-1$ where $k-1\leq d$ is an integer. The upper bound on
the chromatic number is given by coloring $\mathbb{R}$ in sequences
of $n$ half-open intervals of length $1$, where $n-1\geq d$.
\end{proof}

We leave as an exercise to the reader to determine the precise value of $\chi(\mathbb{E}^1,[1,d])$.

\begin{proof}[Proof of Theorem \ref{theo:proper}]
Let $n>1$ be an non-prime integer and
consider on $\mathbb{R}$ the metrics defined by
\[
\rho_i(x,y) = \max\left(\min(|x-y|,1),\frac1{d_i}|x-y|\right)
\]
where $i=1,2$ and the $d_i$ are positive integers such that $(d_1+1)(d_2+1)=n$.
Let $\rho=\rho_1 \stackrel{\infty}{\times} \rho_2$;
It is a proper planar metric and by the above lemmas we have
\begin{eqnarray*}
\chi(\rho,\{1\}) &\leq& \chi(\rho_1,\{1\})\cdot\chi(\rho_2,\{1\}) \\
  &=& \chi(\mathbb{E}^1,[1,d_1])\cdot\chi(\mathbb{E}^1,[1,d_2]) \\
  &=& (d_1+1)(d_2+1)=n
\end{eqnarray*}
and
\begin{eqnarray*}
\chi(\rho,\{1\}) &\geq& \clique(\rho,\{1\}) \\
  &\geq& \clique(\rho_1,\{1\}) \cdot\clique(\rho_2,\{1\}) \\
  &=& (d_1+1)(d_2+1) = n
\end{eqnarray*}
\end{proof}

In view of Theorems \ref{theo:proper_large} and \ref{theo:proper},
the following is an interesting and probably reasonable sub-question of Problem \ref{prob:proper}.
\begin{prob}
Does it hold $\planar^*>3$?
\end{prob}
One could try to adapt the proof of \cite{Chilakamarri} that $\chi(\rho,\{1\})>3$ when
$\rho$ comes from a norm, but as can be seen in the proof of Theorem
\ref{theo:proper_large} the general case can be more tedious.

Last, let us prove our first stated result.
\begin{proof}[Proof of Theorem \ref{theo:planar}]
For any integer $d>0$, consider the metric defined on $\mathbb{R}^2$ by
\[\rho(x,y) = \max\left(\min(|x_1-y_1|,1),\frac1d|x_1-y_1|,\frac{|x_2-y_2|}{1+|x_2-y_2|}\right).\]
It is a (non-proper) planar metric and $\chi(\rho,\{1\})=\chi(\mathbb{E}^1,[1,d])=d+1$.
Propositions \ref{prop:12planar} and \ref{prop:aleph0} conclude the proofs.
\end{proof}

\subsection{Another variation}

In the determination of $\chi(\rho,D)$ where $\rho$ is a proper planar metric and $D$ is
compact, all that matters is of course the union of spheres $\cup_{d\in D} S_\rho(0,d)$,
a compact that does not contain $0$. The metric axioms impose further restriction
on this compact, which we can wave to obtain the following definition.
\begin{defi}[graph and chromatic number of a compact set]
Given a compact set $K$ of the plane $\mathbb{R}^2$ such that $0\notin K$ and $-K=K$, let $G(K)$
be the graph whose vertices are the points of $\mathbb{R}^2$ and where $x\sim y$
whenever $y-x\in K$. This will be called the \emph{graph of $K$}, and its chromatic number
will be denoted by $\chi(K)$ 
\end{defi}
Note that to ensures that $\sim$ is a symmetric relation, we need to ask that $K$ is symmetric;
this together with translation invariance and properness are the sole axioms we kept from
proper planar metrics.

All previous problems are particular cases of the problem of determining $\chi(K)$; for example
$\chi(\mathbb{E}^2,\{1\})=\chi(S_{\mathrm{euc}}(0,1))$. It would therefore be a bit rough
to ask the value of $\chi(K)$ for all $K$.

Let us denote by $\compact$
the set of possible values of $\chi(K)$ when $K$ runs over symmetric compact sets of
the plane avoiding $0$.

Let us now denote by $\compact(2k)$ the possible values of $\chi(K)$ when $K$
is assumed to have $2k$ elements.
Then using the De Bruijn-Erdös theorem (and thus
the axiom of choice), one easily proves that $\compact(2k)\leq 2k+1$, but this bound
is probably not tight. 

\begin{prob}
Determine $\compact(2k)$ in ZFC.
\end{prob}
If one wants to avoid the axiom of choice, one should determine the possible
values of $\sup \chi(G')$ where the sup is over the finite subgraphs $G'$ of $G(K)$
and $K$ is assumed to have $2k$ elements.

\begin{prob}
Determine $\compact(2k)$ in any given
axiom system consistent with ZF?
\end{prob}

\begin{prob}
Is it true that $\mathbb{N}\subset\compact$?
\end{prob}

\begin{prob}
What are the tightest comparisons between $\chi(K_1)$, $\chi(K_2)$ and $\chi(K_1\cup K_2)$,
$\chi(K_1\cap K_2)$ ?
\end{prob}

\section{The hyperbolic plane}\label{sec:hyperbolic}

Let us now consider the case $X=\mathbb{H}^2$, the hyperbolic plane.
We shall from now on use $\rho$ only to denote its metric.

One can simply ask:
\begin{prob}\label{prob:hyperbolic_some}
Determine $\chi(\mathbb{H}^2;\{d\})$ for some $d>0$.
\end{prob}
\begin{prob}
Determine $\chi(\mathbb{H}^2;\{d\})$ for all $d$. 
\end{prob}
Of course even the first of these problem seems difficult.

Note that it is the same to fix the curvature of $X$ to $-1$, as we do,
and let $d$ vary or to fix $d=1$ and consider hyperbolic planes of
various (constant) curvature. As a consequence, the large $d$ regime will
be called the ``high curvature case'' (high meaning highly negative)
and the small $d$ regime the ``low curvature case''.

\subsection{Lower bound}

A lower bound is obtained just like in the Euclidean case:
For all $d>0$, we have $\chi(\mathbb{H}^2;\{d\})\geq 4$.

Let indeed $x$ be a point in 
$\mathbb{H}^2$ and $y,z$ be such that $(x,y,z)$ form an equilateral triangle
of side $d$. Let $x'$ be the other point such that $(x',y,z)$ is an equilateral
triangle. If $G(\mathbb{H}^2;\{d\})$ where $3$-colorable, $x'$ would have to 
be colored in the same color as $x$, say red. Then the whole circle
of center $x$ and radius $\rho(x,x')$ should be colored red.

Here, something is to be checked: for all $d$, $\rho(x,x')>d$. This does hold,
but barely when $d\to\infty$; we omit the (simple) proof. Now, this implies
that two points on this red circle are $d$ apart, a contradiction.

This leaves open the following natural question, raised by Kahle \cite{Kahle}
\begin{prob}[Kahle]\label{prob:hyperbolic5}
Do we have $\chi(\mathbb{H}^2;\{d\})\geq 5$ for at least one explicit $d$?
\end{prob}

\subsection{Hyperbolic checkers}

For upper bounds, we shall adapt the method of Sz\'ekely: another
approach would be to use lattices of large systole, but they are
inefficient because their fundamental domains are necessarily large, and one would
have to use exponentially many colors for large $d$.

We use the half-plane model $\mathbb{H}^2=\{(x,y)\in\mathbb{R}^2 \,|\, y>0\}$
with Riemannian metric $y^{-2}(dx^2+dy^2)$ to design a checkerboard tiling of
the hyperbolic plane.

Let $h,\ell$ be positive reals. First, tile the upper half-plane by
horizontal strips $S_n$ defined by $y\in[e^{nh},e^{(n+1)h})$ for $n\in\mathbb{Z}$.
The boundaries of these strips are not geodesics, but they are horocircles
(limit of geodesic circles) and lie precisely at distance $h$ one to the next.

Next, tile each strip $S_n$ by Euclidean rectangle $R_{n,k}$
defined by $x\in[kre^{nh},(k+1)re^{nh})$
where $k\in\mathbb{Z}$ and $r$ is chosen so that 
\[\rho((x,1),(x+r,1))=\ell.\]
Each $R_{n,k}$ has geodesic vertical sides of hyperbolic length $h$,
its bottom vertices are at hyperbolic distance $\ell$ one from the other,
and its top vertices are at distance $e^{-h}\ell$ one from the other.
Its diameter is easily seen to be bounded above by $\max(\ell,h+e^{-h}\ell)$
and is not realized.

There is some choice left in the definition of this tiling, but we shall
not use this extra liberty here. A tiling $\mathcal{R}$ constructed as above
will be said to be a $(h,\ell)$-checkerboard.

\subsection{Upper bound in high curvature}

First we give a linear bound on the chromatic number of the hyperbolic plane
for large $d$.
\begin{theo}\label{theo:high-curvature}
For all $d\geq 3\ln 3\simeq 3.296$ we have
\[\chi(\mathbb{H}^2,\{d\})\leq 4\left\lceil\frac d{\ln 3}\right\rceil + 4.\]
\end{theo}

\begin{proof}
Let $\mathcal{R}=(R_{n,k})$ be a $(\ln 3,d)$-checkerboard, and color the
tile $R_{n,k}$ by $(n \mod N, k \mod 4)$ where
\[N= \left\lceil\frac d{\ln 3}\right\rceil +1.\]

The diameter of a tile is at most $\max(d,\ln3+d/3)=d$,
and two points of the same color in different tiles
are either in the same slice but in tiles of horizontal indices
$k,k'$ at least $4$ apart, therefore at distance $>(4-1)e^{-\ln3}d=d$,
or in different slices $S_n,S_{n'}$ with $n'\geq n+4$, therefore
at distance $>(N-1)\ln 3\geq d$.
\end{proof}

One wonders what is the true asymptotic behavior of the chromatic number
for large $d$, in particular:
\begin{prob}
Does $\chi(\mathbb{H}^2,\{d\})$ increase linearly in $d$? That is:
do we have
\begin{align*}
\limsup_{d\to\infty} \frac{\chi(\mathbb{H}^2,\{d\})}{d} >0? \\
\liminf_{d\to\infty} \frac{\chi(\mathbb{H}^2,\{d\})}{d} >0?
\end{align*}
\end{prob}

Quantitatively, we ask:
\begin{prob}
Improve the constant $4/\ln3$ in Theorem \ref{theo:high-curvature} or
prove it is best possible.
\end{prob}

In the opposite direction, Matthew Kahle \cite{Kahle} suggested that the chromatic
number could be bounded.
\begin{prob}[Kahle]
Is $\chi(\mathbb{H}^2,\{d\})$ bounded independently of $d$?
\end{prob}

If this turns out not to be the case, then an even more Ramsey-like question
would make much sense.
\begin{prob}\label{prob:Ramsey}
Which finite graphs $H$ appear as (not necessarily induced) monochromatic subgraphs of
all $r$-colorings of $G(\mathbb{H}^2,\{d\})$ for all $d>D(H,r)$?
Give estimates on the function $D$.
\end{prob}
That $\chi(\mathbb{H}^2,\{d\})\to\infty$ when $d\to\infty$
would exactly mean that the graph with two vertices and an edge would
be an example of valid $H$ in Problem \ref{prob:Ramsey}.

\subsection{Upper bound in low curvature}

Let us now give a uniform upper bound for small values of $d$.
\begin{theo}\label{theo:low-curvature}
For all $d\leq 2\ln(3/2) \simeq 0.81$, we have 
\[\chi(\mathbb{H}^2,\{d\})\leq 12.\]
\end{theo}

\begin{proof}
Let $\mathcal{R}=(R_{n,k})$ be a $(d/2,d/2)$-checkerboard, and color the
tile $R_{n,k}$ by $(n \mod 3, k \mod 4)$.

The diameter of a tile is at most $d$,
and two points of the same color in different tiles
are either in the same slice but in tiles of horizontal indices
$k,k'$ at least $4$ apart, therefore at distance $>(4-1)e^{-d/2}d/2\geq d$,
or in different slices $S_n,S_{n'}$ with $n'\geq n+3$, therefore
at distance $>(3-1)d/2= d$.
\end{proof}

Note that by using checkerboard tilings, one can prove that $\chi(\mathbb{H}^2,\{d\})$
is finite for all $d$. One would think that for very small $d$, a $(h,h)$-checkerboard tiling
with $h$ of the order of $d/\sqrt{2}$ could be made to look very much like the Sz\'ekely
tiling. There is a difficulty here, as one cannot arrange the top and bottom
sides to have the same size, and successive slices can therefore not be ``synchronized''.

\begin{prob}
Improve the value $12$ in Theorem \ref{theo:low-curvature} for small enough $d$.
\end{prob}
It is likely that such an improvement is possible, but down to what number can we push this method?
Since at low curvature
the hyperbolic plane looks more and more flat, one would like to relate
the hyperbolic case to the Euclidean one.
\begin{prob}
Does it exist $d_0>0$ such that for all $d\in(0,d_0)$, $\chi(\mathbb{H}^2,\{d\})\leq 7$?
Determine as great a value of $d_0$ as you can.
\end{prob}

\begin{prob}\label{prob:hyp_euc}
Do we have $\chi(\mathbb{H}^2,\{d\})\to \chi(\mathbb{E}^2,\{1\})$ when $d\to 0$?
If we do, what is the greatest $d_1$ such that $\chi(\mathbb{H}^2,\{d\})=\chi(\mathbb{E}^2,\{1\})$
for all $d< d_1$ ?
\end{prob}

Let us stress another question, once again probably ambitious.
\begin{prob}[Kahle]
Is $\chi(\mathbb{H}^2,\{d\})$ non-decreasing in $d$?
\end{prob}

At least, one can wonder whether a partial combination of the last two problems holds.
\begin{prob}
Do we have $\chi(\mathbb{H}^2,\{d\})\geq \chi(\mathbb{E}^2,\{1\})$ for all $d$?
\end{prob}

\subsection{For the road: what about an interval of distances?}

If one wants to estimate for example $\chi(\mathbb{H}^2,[d,d(1+\varepsilon)])$ for a
positive $\varepsilon$, an important phenomenon to take into account
is the \emph{concentration}: in a hyperbolic circle of radius $r\gg 1$,
most pairs of points are at distance roughly $2r$. 
It follows that $G(\mathbb{H}^2,[d,d(1+\varepsilon)])$ contains large cliques
when $d$ is large.
One can more precisely prove that
\[e^{\varepsilon d} \lesssim \clique(\mathbb{H}^2,[d,d(1+\varepsilon)]) 
\leq \chi(\mathbb{H}^2,[d,d(1+\varepsilon)])\]
Given we arrive at the end of our numerotation, we let the last problem
quite vague.
\begin{prob}
Study the behaviour of 
\[\chi(\mathbb{H}^2,[d,d+f(d)]) \quad\mbox{and}\quad
\clique(\mathbb{H}^2,[d,d+f(d)])\]
for various functions $f$.

In particular, can one find a $f$ (to be searched in the $\ln d$ regime)
such that these numbers stay very close one to the other as $d\to\infty$?
\end{prob}

\bibliographystyle{smfalpha}
\bibliography{biblio}

\end{document}